\documentclass[letterpaper]{amsart}

\usepackage[utf8]{inputenc}
\usepackage[T1]{fontenc}
\usepackage{lmodern}
\usepackage{diagrams}

\usepackage{tikz}

\usepackage[pdfusetitle]{hyperref}

\def\arxiv#1{\href{http://arxiv.org/abs/#1}{\texttt{arXiv:#1}}}
\def\doi#1{\href{http://dx.doi.org/#1}{doi:#1}}

\makeatletter

\let\MF@diagram=\diagram
\@namedef{diagram*}{\MF@diagram}
\@namedef{enddiagram*}{\enddiagram}
\def\diagram{\MF@diagram[LaTeXeqno,moreoptions]}

\makeatother

\let\epsilon\varepsilon

\let\newterm\emph
\def\mfbf#1{\hbox{\boldmath{$#1$}}}

\def\mathsmash#1{\hbox to 1em{\boldmath{$\displaystyle #1$\hss}}}

\def\ie{\emph{i.\,e.}}
\def\cf{\emph{cf.}}

\def\Z{\mathbb Z}

\def\Ab{\bar A}
\def\sA{s^{-1} A}

\def\AA{A'\otimes A''}

\def\tmu{\tilde\mu}
\def\bB{\bar B}
\def\Bone{\mathbf{1}} 
\def\cupone{\mathbin{\cup_1}}

\def\HH{\mathcal H}
\def\PP{\mathcal P}
\def\UU{\mathcal U}

\def\Hthree{\mathcal H_{3}}
\def\Ass{\mathcal{A}\mathit{ss}}
\DeclareMathOperator\End{\mathcal{E}\mkern-1.5mu\mathit{nd}}

\DeclareMathOperator{\Tot}{Tot}

\theoremstyle{plain}
\newtheorem{theorem}{Theorem}[section]
\newtheorem{proposition}[theorem]{Proposition}

\theoremstyle{definition}

\newtheorem{remark}[theorem]{Remark}
\newtheorem{example}[theorem]{Example}
\newtheorem{question}[theorem]{Question}

\theoremstyle{remark}
\newtheorem*{acknowledgements}{Acknowledgements}

\numberwithin{equation}{section}

\begin{document}

\title[Tensor products of homotopy Gerstenhaber algebras]{Tensor products of\\homotopy Gerstenhaber algebras}
\author{Matthias Franz}
\thanks{The author was supported by an NSERC Discovery Grant.}
\address{Department of Mathematics, University of Western Ontario,
      London, Ont.\ N6A\;5B7, Canada}
\email{mfranz@uwo.ca}
      
\subjclass[2010]{Primary 57T30; secondary 16E45}

\begin{abstract}
  On the tensor product of two homotopy Gerstenhaber algebras
  we construct a Hirsch algebra structure
  which extends the canonical dg~algebra structure.
  Our result applies more generally
  to tensor products of ``level~$3$ Hirsch algebras''
  and also to the Mayer--Vietoris double complex.
\end{abstract}

\maketitle

\section{Introduction}

Let $R$ be a commutative unital ring
and $A$ an augmented associative differential graded (dg) algebra over~$R$.
A \newterm{Hirsch algebra} structure on~$A$ is a (possibly non-associative)
multiplication in the normalized bar construction~$\bB A$ of~$A$ which is a morphism of
coalgebras and has the counit~$\Bone \in \bB A$
as a 
unit.
It is uniquely determined by its associated twisting cochain
\begin{equation}
  E\colon \bB A\otimes \bB A\to A.
\end{equation}
Because the map~$a_{1}\otimes b_{1}\mapsto E([a_{1}],[b_{1}])$
is essentially a $\cupone$~product for~$A$
(without strict Hirsch formulas),
the product of a Hirsch algebra is always commutative up to homotopy
in the naive sense.

Let $a=[a_{1}|\cdots|a_{k}]\in \bB_{k}A$ and $b=[b_{1}|\cdots|b_{l}]\in \bB_{l}A$.
A Hirsch algebra with $E(a,b)=0$ for all~$k>1$
is called a ``level~$3$ Hirsch algebra'' in~\cite{Kadeishvili:2005}.
It is a \newterm{homotopy Gerstenhaber algebra}
(or \newterm{homotopy G-algebra})
if in addition the resulting multiplication is associative.
Important examples of homotopy Gerstenhaber algebras are
the 
cochain complex of a simplicial set or topological space
\cite{Baues:1981},
the Hochschild cochains of an associative algebra
\cite{Kadeishvili:1988},~%
\cite[Sec.~5.1]{GetzlerJones:1994},~%
\cite{VoronovGerstenhaber:1995}
and the cobar construction of a dg~bialgebra over~$\Z_{2}$ \cite{Kadeishvili:2005}.

Let $A'$ and~$A''$ be two Hirsch algebras.
Then $A'\otimes A''$ is a dg~algebra,
again commutative up to homotopy in the naive sense.
In this paper we address the question of whether such a homotopy
is part of a system of higher homotopies.
We obtain the following result:

\begin{theorem}
  \label{tensor-level-3}
  Let $A'$ and $A''$ be two level~$3$ Hirsch algebras.
  Then $A'\otimes A''$ is a Hirsch algebra in a natural way.
  Moreover, the shuffle map
  $ 
    \bB A'\otimes \bB A''\to \bB(A'\otimes A'')
  $ 
  is multiplicative.
\end{theorem}

The paper is organized as follows:
In Section~\ref{sec:notation} we introduce the notation needed
for the later parts.
The Hirsch algebra structure of~$A=A'\otimes A''$ is constructed
in Section~\ref{sec:construction}. Example~\ref{example-small} shows
how our twisting cochain~$E\colon \bB A\otimes \bB A\to A$
looks like in small degrees, and Example~\ref{example-2-5}
illustrates a general recipe for computing it explicitly.
Section~\ref{sec:proof} contains the proof that $E$ is well-defined
and that the shuffle map is multiplicative.
We conclude by reformulating our result in an operadic language
and applying it to the Mayer--Vietoris double complex
in Section~\ref{sec:operad}.

\begin{acknowledgements}
  The author would like to thank Tornike Kadeishvili for helpful discussions.
\end{acknowledgements}

\section{Notation}
\label{sec:notation}

We work in a cohomological setting, so that differentials are of degree~$+1$.
We denote the desuspension of a complex~$C$ by~$s^{-1}C$,
and the canonical chain map~$s^{-1}C\to C$ of degree~$1$ by~$\sigma$.
Anticipating the definition of the bar construction, we also write
$\sigma^{-1}(c)=[c]$ for~$c\in C$. The differential on~$s^{-1}C$
is given by~$d[c]=-[dc]$.

Let $A$ be an augmented, unital associative dg~algebra over~$R$
with multiplication map
$ 
  \mu_{A}\colon A\otimes A\to A
$ 
and augmentation
$ 
  \epsilon_{A}\colon A\to R
$. 
Denote the augmentation ideal of~$A$ by~$\Ab$,
so that $A=R\oplus\Ab$ canonically.

Note that there are canonical isomorphisms of complexes
\begin{subequations}
  \label{iso-sAA}
  \begin{align}
    \sA'\otimes A'' &\to s^{-1}(A'\otimes A''),
    &
    [a']\otimes a'' &\mapsto [a'\otimes a''], \\
    A'\otimes\sA'' &\to  s^{-1}(A'\otimes A''),
    &
    a'\otimes[a''] &\mapsto (-1)^{|a'|}[a'\otimes a''].
  \end{align}
\end{subequations}

Although we are mostly interested
in the normalized bar construction~$\bB A$ of~$A$,
it will be convenient to consider the unnormalized bar construction~$BA$ as well.
This is is the tensor coalgebra of the desuspension
of~$A$ (instead of~$\Ab$),
\begin{equation}
  BA = T(\sA) = \bigoplus_{k\ge0} (\sA)^{\otimes k}.
\end{equation}
We write $B_{k}A = (\sA)^{\otimes k}$ and
for elements $[a_{1}|\cdots|a_{k}]\in B_{k}A$.
The differential on~$BA$ is the sum of
the tensor product differential~$d_{\otimes}$
and the differential 
\begin{equation}
  \partial = \sum_{i=1}^{k-1}
  1^{\otimes i-1}\otimes\tmu\otimes 1^{\otimes k-i-1}
  \colon B_{k}A \to B_{k-1}A.
\end{equation}
Here $\tmu$ denotes the desuspension of~$\mu$,
\begin{equation}
  \tmu = \sigma^{-1}\mu(\sigma\otimes\sigma)\colon\sA\otimes\sA\to\sA.
\end{equation}
We write $\Bone \in B_{0}A$ for the counit of~$BA$
and $\alpha$ for the canonical twisting cochain
\begin{equation}
  \alpha\colon BA\to B_{1}A=\sA \stackrel{\sigma}\longrightarrow A.
\end{equation}


Let $M$ be a right dg-$A$-module and $N$ a left dg-$A$-module
with structure maps $\mu_{M}\colon M\otimes A\to M$
and $\mu_{N}\colon A\otimes N\to N$, respectively.
The two-sided bar construction of the triple~$(M,A,N)$ is
\begin{equation}
  B(M,A,N) = M\otimes BA\otimes N
\end{equation}
with differential~$d_{B(M,A,N)}=d_{M\otimes BA\otimes N}+\partial'$, where
\begin{equation}
  \label{partial-BAAA}
  \partial' = (\mu_{M}(1\otimes\alpha)\otimes1\otimes1)(1\otimes\Delta\otimes1)
    - (1\otimes1\otimes\mu_{N}(\alpha\otimes1))(1\otimes\Delta\otimes1),
\end{equation}
and with augmentation
\begin{subequations}
\begin{align}
  \epsilon_{B(M,A,N)}\colon B(M,A,N) &\to M\otimes_{A} N, \\
  m[a_{1}|\dots|a_{k}]n &\mapsto \begin{cases} m\otimes n & \text{if $k = 0$}, \\ 0 & \text{otherwise}. \end{cases}
\end{align}
\end{subequations}

We write repeated (co)associative maps in the form
\begin{align}
  \mu^{(k)} &\colon A^{\otimes k}\to A, \\
  \Delta^{(k)}   &\colon T(\sA)\to T(\sA)^{\otimes k},
\end{align}
for instance, and we agree that $\mu^{(0)}$ is the unit map~$\iota\colon R\to A$.

We will also need the concatenation operator
\begin{equation}
  \nabla\colon BA\otimes BA\to BA,
  \quad
  [a_{1}|\cdots|a_{k}]\otimes[b_{1}|\cdots|b_{l}]
  \mapsto
  [a_{1}|\cdots|a_{k}|b_{1}|\cdots|b_{l}],
\end{equation}
which satisfies
\begin{subequations}
\begin{equation}\label{d-concat}
  d(\nabla) = 
  \nabla^{(3)}
  (1\otimes\tmu\otimes1)(1\otimes\alpha\otimes\alpha\otimes1)
  (\Delta\otimes\Delta)
\end{equation}
and
\begin{align}\label{t-concat-diag-l}
  (\alpha\otimes1)\Delta\nabla
  &= (\alpha\otimes\nabla)(\Delta\otimes1)+\epsilon_{BA}\otimes(\alpha\otimes1)\Delta \\
  &= (1\otimes\nabla)\bigl((\alpha\otimes1)\Delta\otimes1\bigr)+\epsilon_{BA}\otimes(\alpha\otimes1)\Delta, \\
  \label{t-concat-diag-r}
  (1\otimes\alpha)\Delta\nabla
  &= (\nabla\otimes\alpha)(1\otimes\Delta)+(1\otimes\alpha)\Delta\otimes\epsilon_{BA} \\
  &= (\nabla\otimes1)\bigl(1\otimes(1\otimes\alpha)\Delta\bigr)+(1\otimes\alpha)\Delta\otimes\epsilon_{BA}.
\end{align}
\end{subequations}

\medbreak

On both the unnormalized and the normalized bar construction,
we will only consider multiplications which are coalgebra maps
and have the counit~$\Bone$ as a (two-sided) unit.
We do not require the multiplication to be associative.

Any such multiplication~$f\colon BA\otimes BA\to BA$
is uniquely determined
by its twisting cochain~$E=\alpha f$, which satisfies
\begin{subequations}
\begin{align}
  d(E) &= E\cup E, \\
  E(\Bone,-)=E(-,\Bone) &= \alpha.
\end{align}
\end{subequations}
We will only consider twisting cochains~$E$ satisfying both conditions.

Any multiplication on the normalized bar construction~$\bB A\subset BA$
can be extended to~$BA$ in a canonical way: Define
$E([1],\Bone )=E(\Bone ,[1])=1$ and,
for $a=[a_{1}|\cdots|a_{k}]$, $b=[b_{1}|\cdots|b_{l}]\in BA$,
set $E(a,b)=0$ if $k+l>1$ and some~$a_{i}=1$ or some~$b_{j}=1$.
Then $E(a,b)\in\Ab$ whenever $k+l>1$.
We call a twisting cochain having these additional properties
\emph{normalized}.
Any normalized twisting cochain~$E\colon BA\otimes BA\to A$
comes from a unique multiplication on~$\bB A$.

For a map~$E\colon BA\otimes BA\to A$
and~$a\in BA$ we define
\begin{equation}
  E_a\colon B A\to A,
  \quad
  b\mapsto E(a, b).
\end{equation}
In this notation, the 
properties 
of a multiplication on~$BA$
become
\begin{subequations}
\label{conditions-Hirsch}
\begin{align}
  \label{condition-Hirsch-differential}
  d( E_a) &= -  E_{d a}
  + \sum_{i=0}^{k} (-1)^{|[a_1|\cdots|a_{i}]|} \mu\bigl( E_{[a_1|\cdots|a_{i}]}
    \otimes  E_{[a_{i+1}|\cdots|a_k]}\bigr)\Delta \\
  \label{condition-Hirsch-1}
   E_{\Bone }(b) &= \alpha(b), \\
  \label{condition-Hirsch-0}
   E_{a}(\Bone ) &= \alpha(a)
\end{align}
\end{subequations}
for $a=[a_{1}|\cdots|a_{k}]$ and~$b=[b_{1}|\cdots|b_{l}]\in BA$.
If $E$ is normalized, then one additionally has
\begin{subequations}
\label{conditions-Hirsch-normalization}
\begin{align}
  \label{condition-Hirsch-normalization-2}
   E_{a}(b) &= 0
  \quad\text{if $k+l>1$ and some~$a_{i}=1$ or some~$b_{j}=1$} \\
  \label{condition-Hirsch-normalization-1}
  \epsilon( E_{a}(b)) &= 0
  \quad\text{if $k+l>1$}.
\end{align}
\end{subequations}  
If $E$ is of level~$3$, then
condition~\eqref{condition-Hirsch-differential} is equivalent
to the two identities
\begin{subequations}
  \label{conditions-level-3}
\begin{align}
  \label{differential-E-hga}
  d( E_{[a_{1}]}) &= -  E_{d[a_{1}]} 
    + \mu\bigl(\alpha\otimes  E_{[a_{1}]} + (-1)^{|a_1|-1}  E_{[a_{1}]}\otimes\alpha\bigr)\Delta, \\
  \label{product-E-hga}
   E_{[a_1 a_2]} &= (-1)^{|a_{1}|-1} \mu( E_{[a_1]}\otimes  E_{[a_2]})\Delta.
\end{align}
\end{subequations}

\section{Construction of the twisting cochain}
\label{sec:construction}

Let $A'$~and~$A''$ be two level~$3$ Hirsch algebras
with twisting cochains $ E'$~and~$ E''$, respectively.
Set $A=\AA$.
%
We are going to inductively define maps
$ 
  G_{a}\colon B A\to B(A,A,A)
$ 
of degree~$|a|+1$ for~$a\in BA$ and then set
$ 
   E_{a} = \epsilon_{B(A,A,A)}G_{a}.
$ 
In Section~\ref{sec:proof} we will show
that this defines a twisting cochain~$ E\colon BA\otimes BA\to A$,
hence a multiplication in~$BA$.
Moreover, if both $E$~and~$E''$ are normalized, then so is $E$.

For the construction as well as for the proof, it is convenient to identify
$B(A,A,A)$ with~$A\otimes BA\otimes A$. This is an isomorphism of graded $R$-modules;
the difference between the two differentials is given by~\eqref{partial-BAAA}.
We write $a=[a_{1}|\cdots|a_{k}]\in BA$ with~$a_{i}=a'_{i}\otimes a''_{i}$.

For~$k=0$ we set $ E_{\Bone }=\alpha$
as required by~\eqref{condition-Hirsch-1}.
We define for~$k=1$
\begin{equation}
  G_{[a_{1}]} = \bigl(( E'_{[a'_{1}]}\otimes\mu_{A''})\otimes1\otimes(\mu_{A'}\otimes E''_{[a''_{1}]})\bigr)\Delta^{(3)}
\end{equation}
and for~$k>1$
\begin{equation}
  G_{a} = M( E'_{[a'_{1}]}, E''_{[a''_{1}]},G_{[a_{2}|\cdots|a_{k}]}).
\end{equation}
Here we have used the abbreviation
\begin{multline}\label{definition-M}
  M(\tilde E',\tilde E'',\tilde G)
  = 
    (1\otimes1\otimes\mu_{A})
    \bigl((\tilde E'\otimes\mu_{A''})\otimes1\otimes(\mu_{A'}\otimes \tilde E'')\otimes1\bigr) \\
    (1\otimes\Delta\nabla^{(3)}\otimes1)
    (1\otimes1\otimes(\sigma^{-1}\otimes1\otimes1)\tilde G)
    \Delta^{(3)}
\end{multline}
for maps $\tilde E'\colon BA'\to A'$, $\tilde E''\colon BA''\to A''$
and $\tilde G\colon B A\to A\otimes B A\otimes A$.
Moreover, by $\tilde E'\otimes\mu_{A''}\colon B A\to A$ we mean the map
\begin{equation}
  [b_{1}|\cdots|b_{k}]\mapsto
  \Bigl(\prod_{i>j}(-1)^{(|b'_{i}|-1)|b''_{j}|}\Bigr)
  \tilde E'([b'_{1}|\cdots|b'_{k}])\otimes\mu_{A''}(b''_{1}\otimes\dots\otimes b''_{k}),
\end{equation}
and similarly by~$\mu_{A'}\otimes \tilde E''\colon B A\to A$
\begin{equation}
  [b_{1}|\cdots|b_{k}]\mapsto
  \Bigl(\prod_{i>j}(-1)^{|b'_{i}|(|b''_{j}|-1)}\Bigr)
  \mu_{A'}(b'_{1}\otimes\dots\otimes b'_{k})\otimes\tilde E''([b''_{1}|\cdots|b''_{k}]).
\end{equation}
By identities~\eqref{iso-sAA}, the differentials of these maps are
\begin{align}
  d(\tilde E'\otimes\mu_{A''}) &= d(\tilde E')\otimes\mu_{A''},
  &
  d(\mu_{A'}\otimes \tilde E'') &= \mu_{A'}\otimes d(\tilde E'').
\end{align}

Figures~\ref{figure-definition-G1} and~\ref{figure-definition-M}
visualize the definitions of $G_{[a_{1}]}$~and
of~$M(\tilde E',\tilde E'',\tilde G)$.

\begin{figure}[ht]
  \centering
\tikzstyle{node} = [draw, fill=black!20, rectangle,
  minimum width=2em, minimum height=2em]
\tikzstyle{box} = [draw, fill=black!20, rectangle,
  minimum width=4em, minimum height=2em]
\begin{tikzpicture}[scale=1]
  \node[draw, circle, fill=black!0] (start) at (3,2) {$B A$};
  \node[node] (a) at (3,-0) {$\Delta^{(3)}$};
  \node[coordinate] (b1) at (1,-1) {};
  \node[coordinate] (b2) at (3,-1) {};
  \node[coordinate] (b3) at (5,-1) {};
  \node[coordinate] (d1) at (1,-2) {};
  \node[coordinate] (d3) at (5,-2) {};
  \node[coordinate] (g1) at (1,-4) {};
  \node[coordinate] (g2) at (3,-4) {};
  \node[coordinate] (g3) at (5,-4) {};
  
  \draw (-0.5,0.75) -- (6.5,0.75) -- (6.5,-2.75) -- (-0.5,-2.75) -- cycle;
  
  \draw (start) -- (a);
  \draw (a) -- (b3)
    -- (d3) node[draw, rectangle, fill=black!20,
      minimum width=3em,minimum height=2em] {$\mu_{A'}\otimes  E''_{[a''_{1}]}$}
    -- (g3) node[draw, circle, fill=white] {$A$};
  \draw (a) -- (b2) -- (g2) node[draw, circle, fill=white] {$B A$};
  \draw (a) -- (b1)
    -- (d1) node[draw, rectangle, fill=black!20,
      minimum width=3em,minimum height=2em] {$ E'_{[a'_{1}]}\otimes\mu_{A''}$}
    -- (g1) node[draw, circle, fill=white] {$A$};
\end{tikzpicture}
  \caption{``Electronic diagram'' for~$G_{[a_{1}]}$}
  \label{figure-definition-G1}
\end{figure}
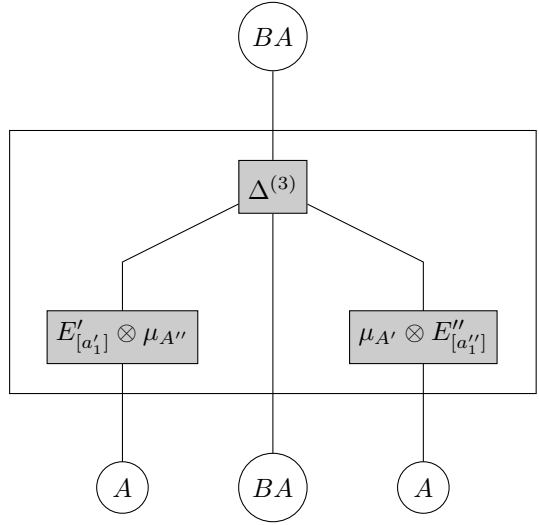

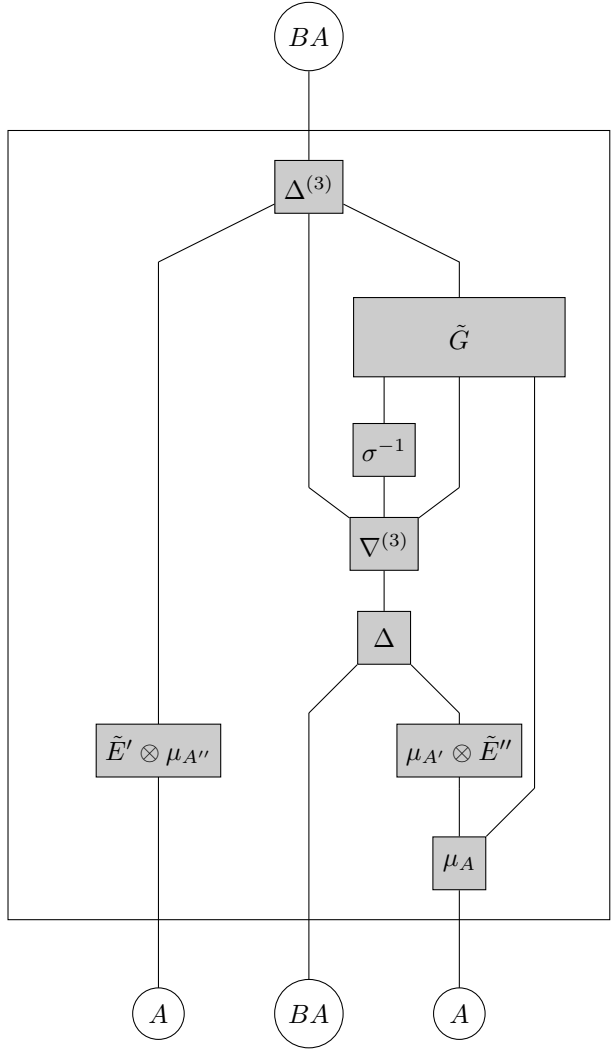
\begin{figure}[ht]
  \centering
\tikzstyle{node} = [draw, fill=black!20, rectangle,
  minimum width=2em, minimum height=2em]
\tikzstyle{box} = [draw, fill=black!20, rectangle,
  minimum width=4em, minimum height=2em]
\begin{tikzpicture}[scale=1]
  \node[draw, , circle, fill=white] (start) at (3,2) {$B A$};
  \node[node] (a) at (3,-0) {$\Delta^{(3)}$};
  \node[coordinate] (b1) at (1,-1) {};
  \node[coordinate] (b2) at (3,-1) {};
  \node[coordinate] (b3) at (5,-1) {};
  \node[coordinate] (G) at (5,-2) {};
  \node[coordinate] (G1) at (4,-2.25) {};
  \node[coordinate] (G2) at (5,-2.25) {};
  \node[coordinate] (G3) at (6,-2.25) {};
  \node[node] (c1) at (4,-3.5) {$\sigma^{-1}$};
  \node[coordinate] (c2) at (5,-4) {};
  \node[coordinate] (c0) at (3,-4) {};
  \node[node] (d1) at (4,-4.75) {$\nabla^{(3)}$};
  \node[node] (d2) at (4,-6) {$\Delta$};
  \node[coordinate] (d0) at (1,-7.5) {};
  \node[coordinate] (e1) at (3,-7) {};
  \node[coordinate] (e2) at (5,-7) {};
  \node[coordinate] (ef) at (5,-7.5) {};
  \node[coordinate] (f1) at (5,-8) {};
  \node[coordinate] (f2) at (6,-8) {};
  \node[node] (gg) at (5,-9) {$\mu_{A}$};
  \node[coordinate] (g1) at (1,-11) {};
  \node[coordinate] (g2) at (3,-11) {};
  \node[coordinate] (g3) at (5,-11) {};
  
  \draw (-1,0.75) -- (7,0.75) -- (7,-9.75) -- (-1,-9.75) -- cycle;r
  
  \draw (start) -- (a);
  \draw (a) -- (b1);
  \draw (a) -- (b2);
  \draw (G1) -- (c1);
  \draw (G2) -- (c2);
  \draw (G3) -- (f2);
  \draw (a) -- (b3) -- (G) node[draw, rectangle, fill=black!20,
    minimum width=8em,minimum height=3em] (GG) {$\tilde G$};
  \draw (b2) -- (c0);
  \draw (c0) -- (d1);
  \draw (c1) -- (d1);
  \draw (c2) -- (d1);
  \draw (d1) -- (d2);
  \draw (d2) -- (e1);
  \draw (d2) -- (e2);
  \draw (e2)
    -- (ef) node[draw, rectangle, fill=black!20,
      minimum width=3em,minimum height=2em] {$\mu_{A'}\otimes\tilde E''$}
    -- (f1);
  \draw (f1) -- (gg);
  \draw (f2) -- (gg);
  \draw (gg) -- (g3) node[draw, circle, fill=white] {$A$};
  \draw (b1)
    -- (d0) node[draw, rectangle, fill=black!20,
      minimum width=3em,minimum height=2em] {$\tilde E'\otimes\mu_{A''}$}
    -- (g1) node[draw, circle, fill=black!0] {$A$};
  \draw (e1) -- (g2) node[draw, circle, fill=white] {$B A$};
\end{tikzpicture}
  \caption{``Electronic diagram'' for~$M(\tilde E',\tilde E'',\tilde G)$}
  \label{figure-definition-M}
\end{figure}


\begin{example}
  \label{example-small}
  The following list shows $E(a,b)$
  for $a\in B_{k}A$ and $b\in B_{l}A$
  with $k\le2$ and $l\le2$. We are ignoring signs here.
  \begin{subequations}
  \begin{align}
    \label{E-1-1}
    E([a_{1}],[b_{1}])
    &= a'_{1}b'_{1}\otimes E''([a''_{1}],[b''_{1}])
    + E'([a'_{1}],[b'_{1}])\otimes b''_{1}a''_{1}, \\
    E([a_{1}],[b_{1}|b_{2}])
    &= a'_{1}b'_{1}b'_{2}\otimes E''([a''_{1}],[b''_{1}|b''_{2}]) \\
    &\quad + E'([a'_{1}],[b'_{1}])b'_{2}\otimes b''_{1}E''([a''_{1}],[b''_{2}]) \\
    &\quad + E'([a'_{1}],[b'_{1}|b'_{2}])\otimes b''_{1}b''_{2}a''_{2}, \\
    E([a_{1}|a_{2}],[b_{1}])
    &= a'_{1}E'([a'_{2}],[b'_{1}])\otimes E''([a''_{1}],[b''_{1}])a''_{2}, \\
    E([a_{1}|a_{2}],[b_{1}|b_{2}])
    &= a'_{1}E'([a'_{2}],[b'_{1}])b'_{2}\otimes E''([a''_{1}],[b''_{1}])E''([a''_{2}],[b''_{2}]) \\
    &\quad + a'_{1}E'([a'_{2}],[b'_{1}])b'_{2}\otimes E''([a''_{1}],[b''_{1}|b''_{2}])a''_{2} \\
    &\quad + a'_{1}E'([a'_{2}],[b'_{1}|b'_{2}])\otimes E''([a''_{1}],[b''_{1}b''_{2}])a''_{2} \\
    &\quad + a'_{1} b'_{1} E'([a'_{2}],[b'_{2}])\otimes E''([a''_{1}],[b''_{1}|b''_{2}])a''_{2} \\
    &\quad + E'([a'_{1}],[b'_{1}])E'([a'_{2}],[b'_{2}])\otimes b''_{1}E''([a''_{1}],[b''_{2}])a''_{2}.
  \end{align}
  \end{subequations}
\end{example}

\begin{example}
  \label{example-2-5}
  We give a general recipe for computing $E(a,b)$
  as in Example~\ref{example-small}.
  To show all features of the algorithm, we illustrate it with
  $a=\bigl[a'_{1}\otimes a''_{1}\bigm|a'_{2}\otimes a''_{2}\bigr]$
  and $b=\bigl[b'_{1}\otimes b''_{1}\bigm|\cdots\bigm|b'_{5}\otimes b''_{5}\bigr]$.
  We are going to explain how to obtain the terms~$c'\otimes c''\in\AA$
  appearing in~$E(a,b)$, again ignoring signs for simplicity.
  
  We start by looking at the component~$c'\in A'$.
  Take $[b'_{1}|\cdots|b'_{l}]$ and cut it into
  $2k$~pieces such that the pieces at positions~$3$,~$5$,~\ldots,~$2k-1$
  have length at least~$1$.
  In our example, one such decomposition is
  \begin{equation}
    [b'_{1}]\otimes[b'_{2}|b'_{3}]\otimes[b'_{4}|b'_{5}]\otimes\Bone.
  \end{equation}
  (The last piece has length~$0$.)
  Now apply $E'_{[a'_{i}]}$ to the $(2i-1)$-th group and then multiply everything
  together:
  \begin{equation}
    E'_{[a'_{1}]}([b'_{1}])\cdot b'_{2} b'_{3}\cdot E'_{[a'_{2}]}([b'_{4}|b'_{5}])\cdot 1
    = E'([a'_{1}],[b'_{1}]) b'_{2} b'_{3} E'([a'_{2}],[b'_{4}|b'_{5}]) = c'.
  \end{equation}
  These are the possible factors~$c'\in A'$ of the terms~$c'\otimes c''$
  appearing in~$E(a,b)$.
  
  \goodbreak
  
  For each such factor, we now describe which factors~$c''\in A''$ appear:
  Switch from primed to doubly primed variables and
  multiply the components within the odd-numbered groups 
  together to obtain
  \begin{equation}
    \bigl[b''_{1} \bigm| b''_{2} \bigm| b''_{3} \bigm| b''_{4} b''_{5} \bigr].
  \end{equation}
  Take the first factor of the tensor product (in the example, $[b''_{1}]$)
  apart. Cut the rest
  \begin{equation}
    \bigl[b''_{2}\bigm| b''_{3} \bigm| b''_{4} b''_{5} \bigr]
  \end{equation}
  into $k$~pieces. Only cuts satisfying the following condition
  are allowed: If some $b'_{j}$ appears as argument to~$E'_{[a'_{i}]}$,
  then the corresponding element~$b''_{j}$ can only appear
  in the $(i-1)$-th piece or earlier.
  In our example, this forces the second piece to be empty,
  hence the first piece is everything. Now plug the $i$-th piece
  into~$E''_{[a''_{i}]}$ and multiply everything together, including
  the first factor we have put apart earlier:
  \begin{equation}
    b''_{1}\cdot E''_{[a''_{1}]}([b''_{2}|b''_{3}|b''_{4} b''_{5}])\cdot E''_{[a''_{1}]}(\Bone)
    = b''_{1} E''([a''_{1}],[b''_{2}|b''_{3}|b''_{4} b''_{5}]) a''_{1} = c''.
  \end{equation}
  Summing up,
  \begin{equation}
    E'([a'_{1}],[b'_{1}]) b'_{2} b'_{3} E'([a'_{2}],[b'_{4}|b'_{5}])
    \otimes b''_{1} E''([a''_{1}],[b''_{2}|b''_{3}|b''_{4} b''_{5}]) a''_{1}
  \end{equation}
  is one term appearing in~$E(a,b)$.
  (There are $70$~terms altogether.)

  The reason for the length condition imposed in the first step
  is the following:
  The recursive definition of~$G_{a}$ together with
  the assignment~$ E_{a}=\epsilon G_{a}$ force
  everything that ``runs through'' $ E'_{[a'_{i}]}\otimes\mu$, $i>1$,
  to ``go through'' some~$\mu\otimes  E''_{[a''_{j}]}$ with~$j<i$ as well.
  Because $( E'_{[a'_{i}]}\otimes\mu)(\Bone)=a'_{i}\otimes1$ and $E''_{[a''_{j}]}(1)=0$,
  the length of the argument of~$E'_{[a'_{i}]}$ must therefore be at least~$1$
  if~$i>1$.
\end{example}

\begin{remark}
The multiplication in~$\bB(A'\otimes A'')$ is not
associative in general,
not even if it is so in $\bB A'$~and~$\bB A''$
(which means that $A'$~and~$A''$ are homotopy Gerstenhaber algebras).
In the latter case one has
\begin{equation}
  \bigl([a]\cdot [b]\bigr)\cdot [c] + [a]\cdot\bigl([b]\cdot[c]\bigr) =
  d(h)([a], [b], [c])
\end{equation}
for $a=a'\otimes a''$,~$b=b'\otimes b''$,~$c=c'\otimes c''\in\AA$
and
\begin{subequations}
\begin{align}
  h([a], [b], [c]) &=
  \bigl[a' E([b'],[c'])\otimes E([a''],[c''|b''])\bigl] \\
  &\quad + \bigl[E([a'],[b'|c'])\otimes E([b''],[c''])\,a''\bigl].
\end{align}
\end{subequations}
(We are again ignoring signs here.)
\end{remark}
  
\begin{question}
  Is $\bB(A'\otimes A'')$ an $A_{\infty}$-algebra
  if $A'$~and~$A''$ are homotopy Gerstenhaber algebras?
\end{question}

\section{Proof of the main result}
\label{sec:proof}

\def\tE{E}
\def\tmu{\mu}

In Section~\ref{sec:construction}
we constructed a map~$G_{a}\colon BA\to B(A,A,A)$
for each~$a\in BA$. They can be assembled into
a map~$G\colon BA\otimes BA\to B(A,A,A)$. We now study its differential.

Denote the left and right action of~$A$
on~$B(A,A,A)$ by $\mu_{L}$~and~$\mu_{R}$, respectively,
and let $\beta$ be the twisting cochain
\begin{equation}
 \beta=\epsilon_{BA}\otimes\alpha_{BA}\colon BA\otimes BA\to R\otimes A=A. 
\end{equation}

\begin{proposition}
  The differential of~$G$ is
  \begin{equation*}
    d(G) = \mu_{L}(\beta\otimes G)\Delta_{BA\otimes BA}
      + \mu_{R}(G\otimes(E-\beta))\Delta_{BA\otimes BA}.
  \end{equation*}
\end{proposition}

\begin{proof}
  We again identify $B(A,A,A)$ with~$A\otimes B A\otimes A$.
  Taking equation~\eqref{partial-BAAA} into account,
  we have to show
  \begin{subequations}
  \begin{align}
    d(G_{a})  &= - G_{d a} \\
    &\quad + (\tmu\otimes1\otimes1)(\alpha\otimes G_{a})\Delta_{BA} \\
    &\quad + \sum_{i=1}^{k} (-1)^{|[a_1|\cdots|a_{i}]|} (1\otimes1\otimes\tmu)
    (G_{[a_1|\cdots|a_{i}]}\otimes \tE_{[a_{i+1}|\cdots|a_k]})\Delta_{BA} \\
    &\quad - (\tmu\otimes1\otimes1)(1\otimes\alpha\otimes 1\otimes 1)(1\otimes\Delta\otimes 1)G_{a} \\
    &\quad + (1\otimes1\otimes\tmu)(1\otimes 1\otimes\alpha\otimes 1)(1\otimes\Delta\otimes 1)G_{a}
  \end{align}
  \end{subequations}
  for all~$a=[a_{1}|\cdots|a_{k}]\in B A$
  We proceed by induction on~$k$.
  Write $\tilde E'=\tE'_{[a'_{1}]}$ and $\tilde E''=\tE''_{[a''_{1}]}$.
  Recall that we have
  \begin{align}
    \bigl|\tE'_{[a'_{1}]}\bigr| &= |a'_{1}|, &
    \bigl|\tE''_{[a''_{1}]}\bigr| &= |a''_{1}|, &
    \bigl|G_{a}\bigr| &= |a|+1.
  \end{align}  

  For~$k=1$, \ie, $a=[a'_{1}\otimes a''_{1}]\in\sA$,
  we have 
  \begin{subequations}
  \begin{align}
    d(G_{a}) &= \bigl((d(\tE'_{[a'_{1}]})\otimes\mu) \otimes1\otimes(\mu\otimes \tE''_{[a''_{1}]})\bigr)\Delta^{(3)} \\
    &\quad + (-1)^{|a'_{1}|} \bigl((\tE'_{[a'_{1}]}\otimes\mu)\otimes1\otimes(\mu\otimes d(\tE''_{[a''_{1}]}))\bigr)\Delta^{(3)} \\
    \intertext{using formula~\eqref{differential-E-hga}} 
    &= - \bigl((\tE'_{d[a'_{1}]}\otimes\mu) \otimes1\otimes(\mu\otimes \tE''_{[a''_{1}]})\bigr)\Delta^{(3)} \\
    &\quad - (-1)^{|a'_{1}|} \bigl((\tE'_{[a'_{1}]}\otimes\mu)\otimes1\otimes(\mu\otimes \tE''_{d[a''_{1}]})\bigr)\Delta^{(3)} \\
    &\quad + (\tmu\otimes1\otimes1)\bigl(\alpha\otimes(\tE'_{[a'_{1}]}\otimes\mu)\otimes1\otimes(\mu\otimes \tE''_{[a''_{1}]})\bigr)\Delta^{(4)} \\
    &\quad + (-1)^{|a'_{1}|-1} (\tmu\otimes1\otimes1)\bigl((\tE'_{[a'_{1}]}\otimes\mu)\otimes\alpha\otimes1\otimes(\mu\otimes \tE''_{[a''_{1}]})\bigr)\Delta^{(4)} \\
    &\quad + (-1)^{|a'_{1}|} (1\otimes1\otimes\tmu)\bigl((\tE'_{[a'_{1}]}\otimes\mu)\otimes1\otimes\alpha\otimes(\mu\otimes \tE''_{[a''_{1}]})\bigr)\Delta^{(4)} \\
    &\quad + (-1)^{|a'_{1}|+|a''_{1}|-1} (1\otimes1\otimes\tmu) \\
    &\quad\quad \bigl((\tE'_{[a'_{1}]}\otimes\mu)\otimes1\otimes(\mu\otimes \tE''_{[a''_{1}]})\otimes\alpha\bigr)\Delta^{(4)} \\
    &= - G_{d a} \\ 
    &\quad + (\tmu\otimes1\otimes1)(\alpha\otimes G_{[a_{1]}})\Delta \\
    &\quad + (-1)^{|[a_{1}]|} (1\otimes1\otimes\tmu)(G_{[a_{1}]}\otimes\tE_{\Bone })\Delta \\
    &\quad - (\tmu\otimes1\otimes1)(1\otimes\alpha\otimes 1\otimes 1)(1\otimes\Delta\otimes 1)G_{a} \\
    &\quad + (1\otimes1\otimes\tmu)(1\otimes 1\otimes\alpha\otimes 1)(1\otimes\Delta\otimes 1)G_{a}.
  \end{align}
  \end{subequations}
  
  For $k>1$, we write $\tilde a = [a_{2}|\cdots|a_{k}]$ and
  $\tilde G = G_{\tilde a}$.
  Then, using definition~\eqref{definition-M},
  \begin{subequations}
  \begin{align}
    \mfbf{d(G_{a})} &= d\bigl(M(\tilde E', \tilde E'', \tilde G)\bigr) \\
    &= M(d(\tilde E'), \tilde E'', \tilde G)
      + (-1)^{|a'_{1}|} M(\tilde E', d(\tilde E''), \tilde G) \\
    &\quad + (-1)^{|a'_{1}|+|a''_{1}|} (1\otimes1\otimes\tmu_{A})
      \bigl((\tilde E'\otimes\mu_{A''})\otimes1\otimes(\mu_{A'}\otimes \tilde E'')\otimes1\bigr) \\
    &\quad \quad (1\otimes\Delta d(\nabla^{(3)})\otimes1)
    (1\otimes1\otimes\sigma^{-1}\otimes1\otimes1)
    (1\otimes1\otimes\tilde G)\Delta^{(3)} \\
    &\quad + (-1)^{|a'_{1}|+|a''_{1}|-1} M(\tilde E', \tilde E'', d(\tilde G)) \\
    \intertext{using \eqref{d-concat}, $\tilde\mu(1\otimes\sigma^{-1})=\sigma^{-1}\mu(\sigma\otimes1)$ and $\tilde\mu(\sigma^{-1}\otimes1)=-\sigma^{-1}\mu(1\otimes\sigma)$}
    &= M(d(\tilde E'), \tilde E'', \tilde G)
    + (-1)^{|a'_{1}|} M(\tilde E', d(\tilde E''), \tilde G) \\
    &\quad + (-1)^{|a_{1}|} M\bigl(\tilde E', \tilde E'', (\tmu\otimes1\otimes1)(\alpha\otimes\tilde G)\Delta\bigr) \\
    &\quad + (-1)^{|a_{1}|-1} M\bigl(\tilde E',\tilde E'',(\tmu\otimes1\otimes1)(1\otimes\alpha\otimes1\otimes1)(1\otimes\Delta\otimes1)\tilde G\bigr) \\
    &\quad + (-1)^{|a_{1}|-1} M(\tilde E', \tilde E'', d(\tilde G)) \; ;
  \end{align}
  \end{subequations}
  \begin{subequations}
  \begin{align}
    \mfbf{G_{d_{\otimes} a}} &= M(\tE'_{d [a'_{1}]}, \tilde E'', \tilde G)
      + (-1)^{|a'_{1}|} M(\tilde E', \tE''_{d[a''_{1}]}, \tilde G) \\
      &\quad + (-1)^{|a_{1}|-1} M(\tilde E', \tilde E'', G_{d_{\otimes}\tilde a}) \; ;
  \end{align}
  \end{subequations}
  \begin{subequations}
  \begin{align}
    \mathsmash{\sum_{i=2}^{k}
    M(\tilde E',\tilde E'',(1\otimes1\otimes\tmu)
    (G_{[a_2|\cdots|a_{i}]}\otimes \tE_{[a_{i+1}|\cdots|a_k]})\Delta)} \\
    &= \sum_{i=2}^{k} (1\otimes1\otimes\tmu)
    (G_{[a_1|\cdots|a_{i}]}\otimes \tE_{[a_{i}|\cdots|a_k]})\Delta \; ;
  \end{align}
  \end{subequations}
  \begin{align}
    \mfbf{M(\tilde E',\tilde E'',G_{\partial\tilde a})} &=
    \sum_{i=2}^{k-1} (-1)^{|[a_{2}|\cdots|a_{i}]|}
      G_{[a_{1}|\cdots|a_{i}a_{i+1}|\cdots|a_{k}]} \; ;
  \end{align}
  \begin{equation}
    \mfbf{M(\tmu(\alpha\otimes \tilde E')\Delta,\tilde E'',\tilde G)} = (\tmu\otimes1\otimes1)(\alpha\otimes G)\Delta \; ;
  \end{equation}
  \begin{equation}
    \mfbf{M(\tilde E',\tmu(\alpha\otimes \tilde E'')\Delta,\tilde G)} =
    (-1)^{|a'_{1}|} (1\otimes1\otimes\tmu)(1\otimes1\otimes\alpha\otimes1)(1\otimes\Delta\otimes1)G \; ;
  \end{equation}
  and
  \begin{subequations}
  \begin{align}
    \mathsmash{M(\tmu(\tilde E'\otimes\alpha)\Delta,\tilde E'',\tilde G)} \\
    &= (-1)^{|a''_{1}|} (\tmu\otimes1\otimes\tmu)
    \bigl((\tilde E'\otimes\mu)\otimes1\otimes1\otimes(\mu\otimes \tilde E'')\otimes1\bigr)
    \\ &\quad\quad
    (1\otimes1\otimes\Delta\nabla^{(3)}\otimes1)
    (1\otimes(\alpha\otimes1)\Delta\otimes (\sigma^{-1}\otimes1\otimes1)\tilde G)\Delta^{(3)} \\
    \intertext{using \eqref{t-concat-diag-l} and the fact that $\tilde G$ maps to $A\otimes BA\otimes A$}
    &= (-1)^{|a''_{1}|} (\tmu\otimes1\otimes1)(1\otimes\alpha\otimes1\otimes1)(1\otimes\Delta\otimes1)M(\tilde E',\tilde E'',\tilde G) \\
    \label{before-k-2}
    &\quad + (-1)^{|a''_{1}|} (1\otimes1\otimes\tmu) \\
    &\quad\quad \bigl(1\otimes1\otimes(\mu\otimes \tilde E'')\otimes1\bigr)(\tmu\otimes\Delta\otimes1)\bigl((\tilde E'\otimes\mu)\otimes \tilde G\bigr)\Delta \\
    \intertext{We consider the case $k=2$ first.}
    &= (-1)^{|a''_{1}|}  (\tmu\otimes1\otimes1)(1\otimes\alpha\otimes1\otimes1)(1\otimes\Delta\otimes1)M(\tilde E',\tilde E'',\tilde G) \\
    &\quad + (-1)^{|a''_{1}|} (1\otimes1\otimes\tmu) \\
    &\quad\quad \bigl(1\otimes1\otimes(\mu\otimes \tilde E'')\otimes1\bigr)(\tmu\otimes\Delta\otimes1)\bigl((\tilde E'\otimes\mu)\otimes G_{[a_{2}]}\bigr)\Delta \\
    \intertext{using \eqref{product-E-hga} in the form $\tmu\bigl((\tE'_{[a'_{1}]}\otimes\mu)\otimes(\tE'_{[a'_{2}]}\otimes\mu)\bigr)\Delta = (-1)^{|a'_{1}|-1}\tE'_{[a'_{1}a'_{2}]}\otimes\mu$}
    &= (-1)^{|a''_{1}|} (\tmu\otimes1\otimes1)(1\otimes\alpha\otimes1\otimes1)(1\otimes\Delta\otimes1)G \\
    &\quad + (-1)^{|a_{1}|+|a'_{1}||a''_{1}|-1}
      (1\otimes1\otimes\tmu)\bigl(1\otimes1\otimes(\mu\otimes \tE''_{[a''_{1}]})\otimes1\bigr) \\
    &\quad\quad  (1\otimes\Delta\otimes1) G_{[a'_{1}a'_{2}\otimes a''_{2}]}\Delta \\
    \intertext{using \eqref{product-E-hga} in the form $\tmu\bigl((\mu\otimes\tE''_{[a''_{1}]})\otimes(\mu\otimes\tE''_{[a''_{2}]})\bigr)\Delta = (-1)^{|a''_{1}|-1}\mu\otimes\tE''_{[a''_{1}a''_{2}]}$}
    &= (-1)^{|a''_{1}|} (\tmu\otimes1\otimes1)(1\otimes\alpha\otimes1\otimes1)(1\otimes\Delta\otimes1)G \\
    &\quad + (-1)^{|a'_{1}|+|a''_{1}||a'_{2}|} G_{[a'_{1}a'_{2}\otimes a''_{1}a''_{2}]} \\
    \intertext{using $a_{1}a_{2} = (-1)^{|a''_{1}||a'_{2}|} a'_{1}a'_{2}\otimes a''_{1}a''_{2}$}
    &= (-1)^{|a''_{1}|} (\tmu\otimes1\otimes1)(1\otimes\alpha\otimes1\otimes1)(1\otimes\Delta\otimes1)G \\
    &\quad + (-1)^{|a'_{1}|} G_{[a_{1}a_{2}]}.
   \intertext{Continuing at~\eqref{before-k-2} for~$k>2$ and using the same identities as before,}
    &= (-1)^{|a''_{1}|} (\tmu\otimes1\otimes1)(1\otimes\alpha\otimes1\otimes1)(1\otimes\Delta\otimes1)G \\
    &\quad + (-1)^{|a_{1}|+|a'_{1}||a''_{1}|-1}
      (1\otimes1\otimes\tmu)\bigl(1\otimes1\otimes(\mu\otimes \tE''_{[a''_{1}]})\otimes1\bigr) \\
    &\quad\quad  (1\otimes\Delta\otimes1) M(\tE'_{[a'_{1}a'_{2}]},\tE''_{[a''_{2}]},G_{[a_{3}|\cdots|a_{k}]}) \\
    &= (-1)^{|a''_{1}|} (\tmu\otimes1\otimes1)(1\otimes\alpha\otimes1\otimes1)(1\otimes\Delta\otimes1)G \\
    &\quad + (-1)^{|a'_{1}|+|a''_{1}||a'_{2}|} M(\tE'_{[a'_{1}a'_{2}]},\tE''_{[a''_{1}a''_{2}]},G_{[a_{3}|\cdots|a_{k}]}) \\
    &= (-1)^{|a''_{1}|} (\tmu\otimes1\otimes1)(1\otimes\alpha\otimes1\otimes1)(1\otimes\Delta\otimes1)G \\
    &\quad + (-1)^{|a'_{1}|} G_{[a_{1}a_{2}|a_{3}|\cdots|a_{k}]}.
  \end{align}
  \end{subequations}
  So the result is the same for all~$k\ge2$.
  
  \begin{subequations}
  \begin{align}
    \mathsmash{M(\tilde E',\tmu(\tilde E''\otimes\alpha)\Delta,\tilde G)} \\
    &= (1\otimes1\otimes\tmu(\tmu\otimes1))
    \bigl((\tilde E'\otimes\mu)\otimes1\otimes(\mu\otimes \tilde E'')\otimes\alpha\otimes1\bigr)
    \\ &\quad\quad
    (1\otimes\Delta^{(3)}\nabla^{(3)}\otimes1)
    (1\otimes1\otimes\sigma^{-1}\otimes1\otimes1)
    (1\otimes1\otimes \tilde G)\Delta^{(3)} \\
    \intertext{using \eqref{t-concat-diag-r}}
    &= - M\bigl(\tilde E',\tilde E'',(1\otimes1\otimes\tmu)(1\otimes 1\otimes\alpha\otimes 1)(1\otimes\Delta\otimes 1)\tilde G\bigr) \\
    &\quad + (1\otimes1\otimes\tmu(1\otimes\tmu))
    \bigl((\tilde E'\otimes\mu)\otimes1\otimes(\mu\otimes \tilde E'')\otimes 1\otimes1\bigr)
    \\ &\quad\quad
    (1\otimes\Delta\otimes1\otimes\epsilon\otimes1)
    (1\otimes1\otimes \tilde G)\Delta^{(3)} \\
    &= - M\bigl(\tilde E',\tilde E'',(1\otimes1\otimes\tmu)(1\otimes 1\otimes\alpha\otimes 1)(1\otimes\Delta\otimes 1)\tilde G\bigr) \\
    &\quad + (1\otimes1\otimes\tmu)(G_{[a_{1}]}\otimes \tE_{\tilde a})\Delta.
  \end{align}
  \end{subequations}
  Putting all terms together finishes the proof.
\end{proof}


\begin{proposition}
  The map~$E\colon BA\otimes BA\to A$
  is a twisting cochain.
  Moreover, if $E'$~and~$E''$ are normalized,
  then so is $E$.
\end{proposition}

\begin{proof}
  To verify \eqref{condition-Hirsch-differential}, we compute:
  \begin{subequations}
  \begin{align}
    d(\tE_{a}) &= \epsilon\, d(G_{a}) \\
    &=  - \epsilon\, G_{da}
    + \tmu(\alpha\otimes\tE_{a})\Delta \\
    &\quad + \sum_{i=1}^{k} (-1)^{|[a_1|\cdots|a_{i}]|} \tmu(\tE_{[a_1|\cdots|a_{i}]}\otimes\tE_{[a_{i+1}|\cdots|a_k]})\Delta \\
    &= - \tE_{d a} 
    + \sum_{i=0}^{k} (-1)^{|[a_1|\cdots|a_{i}]|} \tmu\bigl(\tE_{[a_1|\cdots|a_{i}]}\otimes \tE_{[a_{i+1}|\cdots|a_k]}\bigr)\Delta.
  \end{align}
  \end{subequations}
  Condition~\eqref{condition-Hirsch-1} holds by definition.
  Condition~\eqref{condition-Hirsch-0} holds for~$k=1$
  because $G_{[a_{1}]}(\Bone )=(a'_{1}\otimes1)\otimes \Bone \otimes(1\otimes a''_{1})$.
  For~$k>1$, one similarly has
  $G_{a}(\Bone )\in (A'\otimes1)\otimes BA\otimes A$,
  hence $\epsilon(G_{a}(\Bone))=0$
  by condition~\eqref{condition-Hirsch-0} for~$E''$.
  (This is related to the length condition in Example~\ref{example-2-5}.)
  
  Assume now that $E'$~and~$E''$ are normalized.
  For the proof of~\eqref{condition-Hirsch-normalization-2}
  one inductively shows $G_{a}(b)\in\bigoplus_{m\ge1} A\otimes B_{m}A\otimes A$
  if some~$a_{i}=1$ or some~$b_{j}=1$.
  For~\eqref{condition-Hirsch-normalization-1}, notice that
  the image of~$\tE'_{[a'_{1}]}\otimes\mu$ lies
  in~$\Ab'\otimes A''\subset\Ab$ if~$a'_{1}\in \Ab'$,
  and analogously for~$\mu\otimes\tE''_{[a''_{1}]}$.
  Hence, $\epsilon(G_{a}(b))\in\Ab$
  if $k\ge1$ and~$a_{1}\in\Ab$.
\end{proof}

We now turn to the shuffle maps
\begin{subequations}
\label{shuffle-bar}
\begin{align}
  \nabla\colon BA'\otimes BA'' &\to B(A'\otimes A''), \\
  \nabla\colon \bB A'\otimes \bB A'' &\to \bB(A'\otimes A''),
\end{align}
\end{subequations}
\cf~\cite[Sec.~7.1]{McCleary:2001}.

\begin{proposition}
  The shuffle maps~\eqref{shuffle-bar} are multiplicative.
\end{proposition}

\begin{proof}
  It suffices to consider the unnormalized bar construction.  
  We have to show that the diagram
  \begin{diagram}[small]
    (BA'\otimes BA'')\otimes(BA'\otimes BA'')
    & \rTo^{\nabla\otimes\nabla} & B(A'\otimes A'')\otimes B(A'\otimes A'') \\
    \dTo & & \\
    (BA'\otimes BA')\otimes(BA''\otimes BA'') & & \dTo_{\mu} \\
    \dTo_{\mu\otimes\mu} & & \\
    BA'\otimes BA'' & \rTo^{\nabla} & B(A'\otimes A'')
  \end{diagram}
  commutes. Because all maps are morphisms of coalgebras,
  it is enough to verify that the two associated
  twisting cochains coincide.
  
  Take two elements~$a=a'\otimes a''\in B_{p'}A'\otimes B_{p''}A''$,
  $b=b'\otimes b''\in B_{q'}A'\otimes B_{q''}A''$.
  The twisting cochain of the composition via~$BA'\otimes BA''$
  vanishes unless $p'=q'=0$ or $p''=q''=0$.
  Consider now the twisting cochain of the other composition.
  It follows from properties \eqref{condition-Hirsch-1}~and~\eqref{condition-Hirsch-0}
  and the inductive definition of~$G_{a}$  
  that for~$p'>0$ this twisting cochain vanishes if $p''>0$ or~$q''>0$.
  The case $p''>0$ is analogous. It is therefore enough
  to check the two cases $a=a'\otimes \Bone $, $b=b'\otimes \Bone $
  and $a=\Bone \otimes a''$, $b=\Bone \otimes b''$.
  That both twisting cochains agree follows again inductively from the definition of~$G_{a}$.
\end{proof}

\section{Operadic reformulation}
\label{sec:operad}

It is useful to translate Theorem~\ref{tensor-level-3}
into the language of operads.
Let $\Ass$ be the operad of associative augmented unital $R$-algebras.
We write $\mu\in\Ass(2)$ for the multiplication,
$\epsilon\in\Ass(1)$ for the augmentation and $\iota\in\Ass(0)$
for the unit.
An operad under~$\Ass$ is a morphism of operads~$\Ass\to\PP$.

We define the \newterm{Hirsch operad}~$\HH$
to be the dg~operad under~$\Ass$
generated
by 
operations~$E_{kl}\in\HH(k+l)_{1-k-l}$
subject to the relations~\eqref{conditions-Hirsch}
and~\eqref{conditions-Hirsch-normalization}
(modulo the desuspension)
plus the generators and relations for~$\Ass$.
A Hirsch algebra then is the same as an algebra over~$\HH$.

Let $\Hthree$ be the dg~operad under~$\Ass$
describing level~$3$ Hirsch algebras.
It is the quotient of~$\HH$ by the relations~$E_{k l}=0$ for~$k>1$.
Equivalently,
it is generated by 
operations~$E_{1 k}\in\Hthree(1+k)_{-k}$ and~$E_{0 1}$
subject to the relations \eqref{conditions-Hirsch}
and~\eqref{conditions-Hirsch-normalization}
with~\eqref{condition-Hirsch-differential} replaced
by~\eqref{conditions-level-3}, and
of course again plus the generators and relations for~$\Ass$.


\begin{theorem}\label{thm:H-H3H3}
  The construction in Section~\ref{sec:construction}
  defines a morphism~$f\colon\HH\to\Hthree\otimes\Hthree$
  of dg~operads under~$\Ass$.
\end{theorem}

\begin{proof}
  Let $\PP$ be the free dg~operad under~$\Ass$
  generated by the operations~$E_{kl}$.
  It is clear that our construction defines a morphism
  of dg~operads under~$\Ass$
  \begin{equation}
    \PP\to\Hthree\otimes\Hthree.
  \end{equation}
  Moreover, we know that the relations for~$\HH$ hold
  whenever $\Hthree\otimes\Hthree$
  acts on a tensor product of two $\Hthree$-algebras~$A'$ and~$A''$.
  More precisely, we have proven that the composed map
  \begin{equation}
    \PP\to\Hthree\otimes\Hthree\to \End(A')\otimes\End(A'')
  \end{equation}
  factors through~$\HH$.
  Because $A'$~and~$A''$ can be free $\Hthree$-algebras
  (\cf~\cite[Sec.~I.1.4]{MarklShniderStasheff:2002}),
  this implies that the necessary relations hold already
  in~$\Hthree\otimes\Hthree$.
\end{proof}

\begin{example}
  The homotopy Gerstenhaber algebra structure
  on the cochain complex~$C^{*}(X)$
  of a simplicial set~$X$ is constructed by dualizing
  a ``homotopy Gerstenhaber coalgebra'' structure
  on the chain complex~$C(X)$.
  Therefore, for simplicial sets $X$~and~$Y$
  there is a natural action of~$\Hthree\otimes\Hthree$
  on the 
  complex dual to~$C(X)\otimes C(Y)$,
  and the canonical map
  \begin{equation}
    C^{*}(X)\otimes C^{*}(Y) \to \bigl(C(X)\otimes C(Y)\bigr)^{*}
  \end{equation}
  is a morphism of~$\Hthree\otimes\Hthree$-algebras, hence
  of $\HH$-algebras.  
  Note however that the dual of the shuffle map
  \begin{equation}
    C^{*}(X\times Y)
    \stackrel{\nabla^{*}}\longrightarrow
    \bigl(C(X)\otimes C(Y)\bigr)^{*}
  \end{equation}
  is \emph{not} a morphism of Hirsch algebras.
  ($\nabla^{*}$ already fails to commute with the operation~\eqref{E-1-1}.)
\end{example}
  
An analogous remark applies to Hochschild cochains.

\begin{example}
  Let $A$ be a cosemisimplicial $\Hthree$-algebra.
  By this we mean a collection~$A^{q}$, $q\ge0$, of
  $\Hthree$-algebras together with morphisms
  $d_{i}\colon A^{q}\to A^{q+1}$, $0\le i\le q+1$, satisfying
  the usual coface relations,
  \cf~\cite[Def.~8.40]{McCleary:2001}.
  Then the associated total complex~$\Tot A^{*}$ is an algebra over~$\Hthree\otimes\Hthree$
  in the following way: Let $E\otimes E'\in\Hthree(m)_{n}\otimes\Hthree(m)_{n'}$,
  and $a_{i}\in A^{q_{i}}$ for~$1\le i\le m$.
  Set $q=\sum_{i}q_{i}-n'$.
  Via the coface operators,
  $E'$ determines morphisms~$\phi_{i}\colon A^{q_{i}}\to A^{q}$
  in the same way as it acts on the unnormalized cochains of a simplicial set.
  We can therefore set
  \begin{equation}
    (E\otimes E')(a_{1},\dots,a_{m})
    = E\bigl(\phi_{1}(a_{1}),\dots,\phi_{m}(a_{m})\bigr)\in A^{q}.
  \end{equation}
  (If $q<q_{i}$ for some~$i$, we define the result to be~$0$.)
  
  An important special case of this is
  the Mayer--Vietoris double complex
  \begin{equation}
    C^{pq}(\UU) = \prod_{i_{0}<\dots<i_{q}}C^{p}(U_{i_{0}}\cap\dots\cap U_{i_{q}};R)
  \end{equation}
  associated to an ordered cover~$\UU=(U_{i})_{i\in I}$ of a simplicial set,
  \cf~\cite[\S\S 8,~14]{BottTu:1982} for instance.
  In this case Theorem~\ref{thm:H-H3H3} says
  that $\Tot C^{* *}(\UU)$ has the structure of a Hirsch algebra
  which extends the familiar dg~algebra structure.
  Note also that the canonical inclusion map
  \begin{equation}
    C^{*}(X;R) \to \Tot C^{* *}(\UU),
    \quad
    \alpha \mapsto \bigl(\left.\alpha\right|_{U_{i}}\bigr)_{i\in I}\in C^{* 0}(\UU;R)
  \end{equation}
  is a morphism of Hirsch algebras
  because for~$n'>0$ the maps~$\phi_{1}$,~\ldots,~$\phi_{m}$
  described above vanish on the image of the inclusion map,
  and for~$n'=0$ they must all be the identity map.
\end{example}

\begin{remark}
  Assume $R=\Z_{2}$ and let $\tau=(12)\in S_{2}$.
  Note that $\mu$ is basis of~$\Hthree(2)_{0}$ over~$R[S_{2}]$,
  and $E_{11}$ is one for~$\Hthree(2)_{-1}$.
  A direct computation shows that up to applying~$\tau$
  and transposing the factors,
  $h=\mu\otimes E_{11} + E_{11}\otimes\tau\mu\in(\Hthree\otimes\Hthree)(2)_{-1}$
  is the only solution to~$d(h)=\mu\otimes\mu + \tau\mu\otimes\tau\mu$.
  Hence, our definition~\eqref{E-1-1} of~$f(E_{11})$
  is essentially the only possible choice.
  Together with~$d(f(E_{21}))\ne0$, this also proves that one cannot hope
  for a morphism~$\Hthree\to\Hthree\otimes\Hthree$ of dg~operads under~$\Ass$
  because condition~\eqref{product-E-hga} never holds.
\end{remark}

But of course one may ask:

\begin{question}
  Is $\HH$ a dg~Hopf operad under~$\Ass$? In other words,
  is the tensor product of two Hirsch algebras again a Hirsch algebra?
\end{question}

\end{document}